\documentclass{article}

\usepackage{arxiv}

\usepackage[utf8]{inputenc} % allow utf-8 input
\usepackage[T1]{fontenc}    % use 8-bit T1 fonts
\usepackage{hyperref}       % hyperlinks
\usepackage{url}            % simple URL typesetting
\usepackage{booktabs}       % professional-quality tables
\usepackage{amsfonts}       % blackboard math symbols
\usepackage{nicefrac}       % compact symbols for 1/2, etc.
\usepackage{microtype}      % microtypography
\usepackage{lipsum}
\usepackage{graphicx}
\usepackage{amsmath,amssymb,amsthm,mathtools}
\usepackage{enumitem}
\usepackage{subcaption}
\usepackage{xcolor}
\newtheorem{theorem}{Theorem}[section]
\newtheorem{proposition}[theorem]{Proposition}
\newtheorem{lemma}[theorem]{Lemma}
\newtheorem{corollary}[theorem]{Corollary}
\theoremstyle{definition}
\newtheorem{definition}[theorem]{Definition}
\newtheorem{example}[theorem]{Example}
\theoremstyle{remark}
\newtheorem{remark}[theorem]{Remark}

\newcommand{\Hcal}{\mathcal{H}}
\newcommand{\Ecal}{\mathcal{E}}
\newcommand{\Vcal}{\mathcal{V}}
\newcommand{\Fcal}{\mathcal{F}}

\newcommand{\orb}{\operatorname{Orb}}
\newcommand{\val}{\operatorname{val}}
\graphicspath{ {./figures/} }

\title{Eulerian and Bipartite Partial Duals of Hypermaps}

\author{
 Yufan Han \\
  College of Mathematical Sciences\\
  Xinjiang Normal University\\
  Urumqi 830017, China \\
  \texttt{} \\
   \And
 Metrose Metsidik \\
  College of Mathematical Sciences\\
  Xinjiang Normal University\\
  Urumqi 830017, China \\
  \texttt{metrose@xjnu.edu.cn} \\
}

\begin{document}
\maketitle
\begin{abstract}
We study hyperedge partial duals of finite hypermaps in a purely combinatorial framework, without assuming orientability. A hypermap is represented by three fixed-point-free involutions $(\tau_0,\tau_1,\tau_2)$ on its flag set. We first give an explicit construction of the medial map from this model: $02$-orbits become the medial vertex discs, while $\tau_1$-transpositions become the medial bands; a local orientation system and its twist data then provide a signed rotation description of the medial map. We next prove that the state circles associated with a chosen set of hyperedges are in natural bijection with the vertex orbits of the corresponding partial dual, yielding a crossing-total characterization of all Eulerian hyperedge partial duals. For bipartiteness, the twist data lead to a modified medial map in which inserted bars record the obstruction to a global orientation. We prove that a partial dual is bipartite if and only if its dualized hyperedge set is exactly the set of $c$-type hyperedges identified by an all-crossing orientation of this modified medial map. When the hypermap is orientable, these constructions specialize to the known orientable-hypermap results; when every hyperedge has valence two, they specialize to the ribbon-graph results.
\end{abstract}

\section{Introduction}

Partial duality is a local form of geometric duality. Chmutov introduced the notion for ribbon graphs~\cite{Chmutov2009}, and Chmutov and Vignes-Tourneret subsequently extended it to hypermaps~\cite{ChmutovVignes2022}. Structural studies of partial duals of ribbon graphs include the work of Moffatt~\cite{Moffatt2011,Moffatt2013}. For ribbon graphs, Huggett and Moffatt characterized bipartite partial duals of plane graphs by all-crossing orientations of the medial graph~\cite{HuggettMoffatt2013}. Metsidik and Jin later gave an Eulerian characterization for plane graphs~\cite{MetsidikJin2018}; Deng, Jin, and Metsidik extended both orientation characterizations to arbitrary ribbon graphs~\cite{DengJinMetsidik2020}. Their work shows that the modified medial graph is indispensable in the non-orientable bipartite case.

Combinatorial models for hypermaps go back to work of Cori, Walsh, Jones, Singerman, and others; background can be found in~\cite{Cori1975,Walsh1975,JonesSingerman1978,LandoZvonkin2004}. The framework encoding maps on surfaces by flags, fixed-point-free involutions, and edge-coloured graphs appears in~\cite{Lins1982,Vince1983}. Standard background on maps on surfaces, ribbon graphs, and their ribbon representations is given in~\cite{GrossTucker1987,EllisMonaghanMoffatt2013}. For orientable hypermaps, the permutation model $(\sigma,\alpha)$ conveniently describes hyperedge partial duality and medial maps; Cori and Hetyei gave a permutation construction of the medial map in that setting~\cite{CoriHetyei2025}. Han and Metsidik established orientation characterizations of Eulerian and bipartite hyperedge partial duals in this model~\cite[Theorems~5.3 and~6.5]{HanMetsidik2026}. However, the left--right $x/y$ labels used by the $(\sigma,\alpha)$-model and its bipartiteness proof depend on a global orientation of the underlying surface.

Our aim is to formulate both orientation characterizations uniformly for all hypermaps. We therefore use the $\tau$-model, also called the dart or flag model. It records the complete local cell structure and applies equally well on non-orientable surfaces. In particular, we give a direct construction of the medial map from this model and use local orientation systems to record how bands are attached relative to vertex discs.

The two problems have different sensitivities to orientability. The Eulerian characterization is controlled by alternating signs on state circles and requires no global left--right convention. The usual proof of the bipartite characterization does require such a convention. We replace it by local orientation systems on medial edges and their twist data; equivalently, every medial edge whose endpoint local orientations are incompatible is subdivided by an artificial bivalent bar. This is precisely the hypermap analogue of the modified medial map in~\cite{DengJinMetsidik2020}. Related background on twists and local duality data for non-orientable embedded graphs can be found in~\cite{EllisMonaghanMoffatt2012}.

The combinatorial relationship between Eulerian maps and hypermaps on orientable surfaces was studied by Jackson and Visentin~\cite{JacksonVisentin1999}. Here we study hyperedge partial duals of hypermaps on arbitrary surfaces.

All hypermaps in this paper are finite. To simplify notation, theorem statements and proofs are given for connected hypermaps; for disconnected hypermaps, the conclusions hold componentwise.

\section{Hypermaps and Hyperedge Partial Duals}

\subsection{The Flag Model}

\begin{definition}\label{def:tau-hypermap}
A \emph{hypermap} is a triple
\[
  \Hcal=(X;\tau_0,\tau_1,\tau_2),
\]
where \(X\) is a finite set, each \(\tau_i\) is a fixed-point-free involution on \(X\), and the group \(\langle\tau_0,\tau_1,\tau_2\rangle\) generated by these involutions acts transitively on \(X\). The elements of \(X\) are called \emph{flags}. For \(i,j\in\{0,1,2\}\), let
\[
  X/\langle\tau_i,\tau_j\rangle
\]
denote the set of orbits of the subgroup \(\langle\tau_i,\tau_j\rangle\) generated by \(\tau_i\) and \(\tau_j\). The sets of vertices, hyperedges, and faces of \(\Hcal\) are defined by
\[
  \Vcal(\Hcal)=X/\langle\tau_1,\tau_2\rangle,\qquad
  \Ecal(\Hcal)=X/\langle\tau_0,\tau_2\rangle,\qquad
  \Fcal(\Hcal)=X/\langle\tau_0,\tau_1\rangle,
\]
respectively.
\end{definition}

Geometrically, a \emph{flag} is a small triangle in the barycentric subdivision of the cell complex associated with \(\Hcal\). The involution \(\tau_i\) crosses the colour-\(i\) line, as shown in Figure~\ref{fig:tau-involutions}, where blue, red, and black represent \(0\), \(1\), and \(2\), respectively. Thus, at a flag \((v,e,f)\), the involutions \(\tau_0\), \(\tau_1\), and \(\tau_2\) change the vertex, hyperedge, and face components of the flag, respectively, while preserving the other two components. This convention agrees with~\cite[Section~1.2]{ChmutovVignes2022}.

\begin{figure}[htbp]
  \centering
  \includegraphics[width=0.60\textwidth]{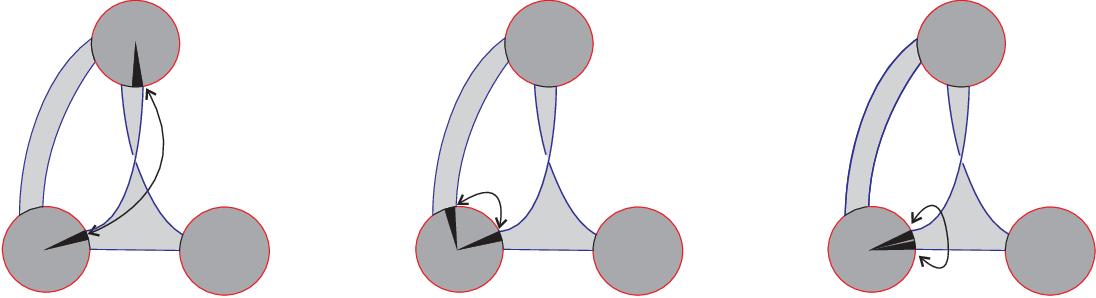}
  \put(-239,35){\footnotesize$\tau_0$}
  \put(-153,25){\footnotesize$\tau_1$}
  \put(-37,16){\footnotesize$\tau_2$}
  \caption{The three basic involutions in the $\tau$-model.}
  \label{fig:tau-involutions}
\end{figure}

Every orbit of two fixed-point-free involutions is an alternating cycle and hence has even cardinality. For \(v\in\Vcal(\Hcal)\), \(e\in\Ecal(\Hcal)\), and \(f\in\Fcal(\Hcal)\), define the corresponding valences by
\[
  \val(v)=\frac{|v|}{2},\qquad
  \val(e)=\frac{|e|}{2},\qquad
  \val(f)=\frac{|f|}{2}.
\]
The hypermap $\Hcal$ is called \emph{Eulerian} if every vertex has even valence.

We fix the following graph-theoretic definition of bipartiteness. The \emph{vertex-adjacency multigraph} $\Gamma(\Hcal)$ has vertex set $\Vcal(\Hcal)$; for every $\tau_0$-transposition $\{x,\tau_0x\}$, it has an edge whose endpoints are the vertex orbits containing $x$ and $\tau_0x$. Loops and multiple edges are allowed.

\begin{definition}\label{def:bipartite}
The hypermap $\Hcal$ is called \emph{bipartite} if $\Gamma(\Hcal)$ is bipartite. This agrees with the usual convention for bipartite hypermaps; see~\cite{Walsh1975,BredaDuarte2007}. Equivalently, there is a map $\chi:\Vcal(\Hcal)\to\{+1,-1\}$ such that
\[
  \chi(V(x))=-\chi(V(\tau_0x))\qquad(x\in X),
\]
where $V(x)$ denotes the $\langle\tau_1,\tau_2\rangle$-orbit containing $x$.
\end{definition}

For an ordinary map, every geometric edge is represented here by two parallel adjacency edges. This does not affect bipartiteness. In the orientable $(\sigma,\alpha)$-model, a $\tau_0$-transposition is exactly the adjacency $v(i)$--$v(\alpha(i))$ used in the usual definition of a bipartite hypermap.

\subsection{Hyperedge Partial Duality}

Let $A\subseteq\Ecal(\Hcal)$ be a set of hyperedges, and let
\[
  X_A=\bigcup_{e\in A}e
\]
be the union of their flag sets. Since $X_A$ is invariant under $\tau_0$ and $\tau_2$, the following definition is unambiguous.

\begin{definition}\label{def:partial-dual}
The \emph{partial dual} of $\Hcal=(X;\tau_0,\tau_1,\tau_2)$ with respect to $A$ is
\[
  \Hcal^A=(X;\tau_0^A,\tau_1,\tau_2^A),
\]
where, for every $x\in X$,
\[
 \tau_0^A(x)=
 \begin{cases}
   \tau_2(x),&x\in X_A,\\
   \tau_0(x),&x\notin X_A,
 \end{cases}
 \qquad
 \tau_2^A(x)=
 \begin{cases}
   \tau_0(x),&x\in X_A,\\
   \tau_2(x),&x\notin X_A.
 \end{cases}
\]
\end{definition}

\begin{lemma}\label{lem:partial-dual-hypermap}
For every $A\subseteq\Ecal(\Hcal)$, the triple
\[
  \Hcal^A=(X;\tau_0^A,\tau_1,\tau_2^A)
\]
is again a connected hypermap.
\end{lemma}

\begin{proof}
Since $X_A$ is a union of $\langle\tau_0,\tau_2\rangle$-orbits, both it and its complement $X\setminus X_A$ are invariant under $\tau_0$ and $\tau_2$. Hence, on each of $X_A$ and $X\setminus X_A$, the restrictions of $\tau_0^A$ and $\tau_2^A$ are just the restrictions of the original involutions $\tau_0$ and $\tau_2$, possibly with their names interchanged. Thus $\tau_0^A$ and $\tau_2^A$ are fixed-point-free involutions.

View the transpositions of the three involutions as edges of colours $0,1,2$ in the flag graph. For every $x\in X$,
\[
 \bigl\{\tau_0^A(x),\tau_2^A(x)\bigr\}
 =
 \bigl\{\tau_0(x),\tau_2(x)\bigr\}.
\]
Thus the flag graphs of $\Hcal^A$ and $\Hcal$ have the same uncoloured underlying graph: the partial dual only swaps colours $0$ and $2$ inside the $02$-bubbles belonging to $X_A$, while colour-$1$ edges remain unchanged. Since the original flag graph is connected, so is the new one. Consequently $\langle\tau_0^A,\tau_1,\tau_2^A\rangle$ acts transitively on $X$, and $\Hcal^A$ is a connected hypermap.
\end{proof}

Thus the partial dual swaps colours $0$ and $2$ precisely in the selected $02$-bubbles and leaves colour $1$ unchanged. This is the hyperedge version of the $\tau$-model formula of Chmutov and Vignes-Tourneret; the direct flag formula is~\cite[Theorem~2.6]{ChmutovVignes2022}, and the colour-swapping formulation in the coloured flag graph is~\cite[Theorem~2.8]{ChmutovVignes2022}. The pointwise form in Definition~\ref{def:partial-dual} avoids ambiguities caused by different conventions for permutation composition.

\begin{proposition}\label{prop:basic-pd}
For any $A,B\subseteq\Ecal(\Hcal)$, the following hold:
\begin{enumerate}[label=\textup{(\roman*)}]
\item The hyperedge orbits of $\Hcal^A$ are naturally the same as those of $\Hcal$, and hyperedge valences are preserved;
\item $(\Hcal^A)^B=\Hcal^{A\triangle B}$; in particular, $(\Hcal^A)^A=\Hcal$.
\end{enumerate}
\end{proposition}

\begin{proof}
On each hyperedge orbit, Definition~\ref{def:partial-dual} either leaves the pair $(\tau_0,\tau_2)$ unchanged or swaps its two members. This proves (i). Performing the same swap twice cancels it, giving (ii).
\end{proof}

\section{Medial Maps and States}

We fix the following running example for the remainder of the paper. 
\begin{example}
Let $X=\{1,2,\ldots,16\}$, and define
\[
\begin{aligned}
\tau_0&=(1\ 8)(2\ 3)(4\ 5)(6\ 7)(9\ 12)(10\ 11)(13\ 16)(14\ 15),\\
\tau_1&=(1\ 14)(2\ 13)(3\ 12)(4\ 8)(5\ 10)(6\ 16)(7\ 11)(9\ 15),\\
\tau_2&=(1\ 2)(3\ 4)(5\ 6)(7\ 8)(9\ 10)(11\ 12)(13\ 14)(15\ 16).
\end{aligned}
\]
Denote the resulting hypermap by $\Hcal_0=(X;\tau_0,\tau_1,\tau_2)$. Its three vertices are represented by the cyclic orders
\[
v=(1,2,13,14),\qquad
u=(3,4,8,7,11,12),\qquad
w=(5,6,16,15,9,10),
\]
its three hyperedges are
\[
h=\{1,2,3,4,5,6,7,8\},\qquad
a=\{9,10,11,12\},\qquad
b=\{13,14,15,16\},
\]
and its unique face is $X$. The cycles on the vertex discs are read counterclockwise; hence the three discs in Figure~\ref{fig:example-hypermap} correspond to $v$, $u$, and $w$, with the cyclic orders listed above. The valences of $h$, $a$, and $b$ are $4$, $2$, and $2$, respectively.
\end{example}

\begin{figure}[htbp]
  \centering
  \setlength{\unitlength}{0.001\textwidth}
  \begin{picture}(620,398)
    \put(0,0){\includegraphics[width=0.62\textwidth]{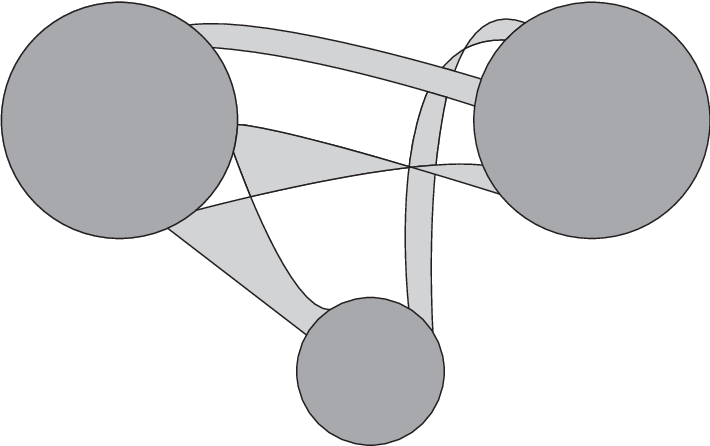}}
    \put(104,284){\makebox(0,0)[c]{\footnotesize$u$}}
    \put(324,65){\makebox(0,0)[c]{\footnotesize$v$}}
    \put(516,284){\makebox(0,0)[c]{\footnotesize$w$}}
    \put(287,105){\makebox(0,0)[c]{\scriptsize$1$}}
    \put(274,94){\makebox(0,0)[c]{\scriptsize$2$}}
    \put(370,90){\makebox(0,0)[c]{\scriptsize$13$}}
    \put(350,110){\makebox(0,0)[c]{\scriptsize$14$}}
    \put(143,196){\makebox(0,0)[c]{\scriptsize$3$}}
    \put(162,214){\makebox(0,0)[c]{\scriptsize$4$}}
    \put(190,250){\makebox(0,0)[c]{\scriptsize$8$}}
    \put(193,276){\makebox(0,0)[c]{\scriptsize$7$}}
    \put(178,337){\makebox(0,0)[c]{\scriptsize$11$}}
    \put(160,355){\makebox(0,0)[c]{\scriptsize$12$}}
    \put(426,248){\makebox(0,0)[c]{\scriptsize$5$}}
     \put(442,220){\makebox(0,0)[c]{\scriptsize$6$}}
    \put(426,290){\makebox(0,0)[c]{\scriptsize$10$}}
    \put(426,310){\makebox(0,0)[c]{\scriptsize$9$}}
    \put(442,345){\makebox(0,0)[c]{\scriptsize$15$}}
    \put(460,360){\makebox(0,0)[c]{\scriptsize$16$}}
    \put(365,200){\makebox(0,0)[c]{\footnotesize$b$}}
    \put(315,330){\makebox(0,0)[c]{\footnotesize$a$}}
    \put(205,166){\makebox(0,0)[c]{\footnotesize$h$}}
  \end{picture}
  \caption{The hypermap $\Hcal_0$.}
  \label{fig:example-hypermap}
\end{figure}

\subsection{The Medial Map in the $\tau$-Model}

We now construct the medial map entirely within the $\tau$-model. It is important to distinguish the flag set of the original hypermap from the flag set of the medial map itself: the former becomes the label set of medial half-edges in the latter, but the two sets are not identical.

Let $\mathsf b$ and $\mathsf w$ be two symbols, and let $\bar c$ denote the symbol different from $c\in\{\mathsf b,\mathsf w\}$. Set
\[
  Y=X\times\{\mathsf b,\mathsf w\}.
\]
Write $Y_{\mathsf b}=X\times\{\mathsf b\}$ and $Y_{\mathsf w}=X\times\{\mathsf w\}$. Define three involutions on $Y$ by
\[
\begin{aligned}
 \mu_0(x,c)&=(\tau_1x,c),\\
 \mu_2(x,c)&=(x,\bar c),\\
 \mu_1(x,\mathsf b)&=(\tau_2x,\mathsf b),\\
 \mu_1(x,\mathsf w)&=(\tau_0x,\mathsf w).
\end{aligned}
\tag{M}\label{eq:tau-medial}
\]

\begin{definition}\label{def:tau-medial}
Let $\Hcal=(X;\tau_0,\tau_1,\tau_2)$ be a hypermap. Its \emph{$\tau$-model representation of the medial map} is defined by
\[
 M_\tau(\Hcal)=(Y;\mu_0,\mu_1,\mu_2),
\]
where $Y$ and the $\mu_i$ are given by~\eqref{eq:tau-medial}. In what follows we write $M(\Hcal)$ for $M_\tau(\Hcal)$.
\end{definition}

\begin{example}
The medial map $M(\Hcal_0)$ of the hypermap $\Hcal_0$ has three medial vertices $v_h$, $v_a$, and $v_b$, corresponding to the hyperedges $h$, $a$, and $b$, whose counterclockwise half-edge orders are
\[
\begin{aligned}
v_h&=(\bar{1},\bar{2},\bar{3},\bar{4},\bar{5},\bar{6},\bar{7},\bar{8}),\qquad
v_a=(\bar{9},\bar{10},\bar{11},\bar{12}),\qquad
v_b=(\bar{13},\bar{14},\bar{15},\bar{16}),
\end{aligned}
\]
where $\bar x=\{x_{\mathsf b},x_{\mathsf w}\}$. It has eight medial edges
\[
m_{1,14},\,m_{2,13},\,m_{3,12},\,m_{4,8},\,
m_{5,10},\,m_{6,16},\,m_{7,11},\,m_{9,15},
\]
where $m_{4,8}$ is a loop at $v_h$; $m_{1,14}$, $m_{2,13}$, and $m_{6,16}$ join $v_h$ to $v_b$; $m_{3,12}$, $m_{5,10}$, and $m_{7,11}$ join $v_h$ to $v_a$; and $m_{9,15}$ joins $v_a$ to $v_b$.
Figure~\ref{fig:example-medial} depicts the three medial vertices and eight medial edges in this order. 
\end{example}

\begin{figure}[htbp]
  \centering
  \setlength{\unitlength}{0.001\textwidth}
  \begin{picture}(720,541)
    \put(0,0){\includegraphics[width=0.72\textwidth]{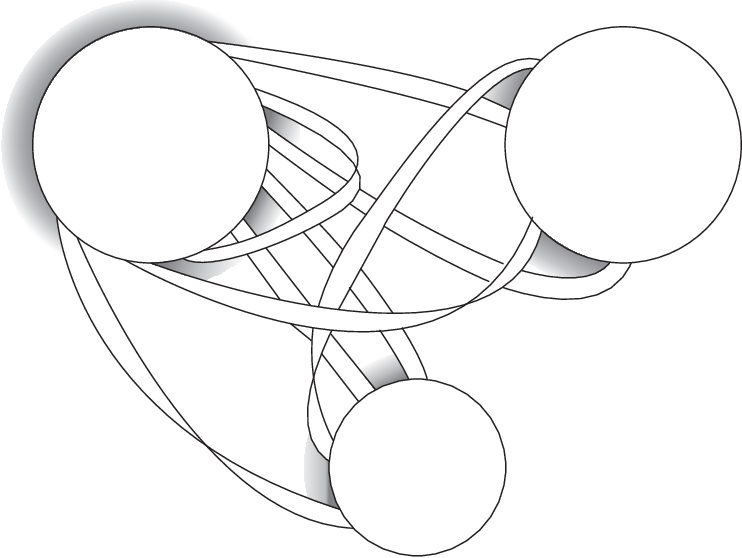}}
    \put(140,408){\makebox(0,0)[c]{\footnotesize$v_h$}}
    \put(600,408){\makebox(0,0)[c]{\footnotesize$v_a$}}
    \put(396,90){\makebox(0,0)[c]{\footnotesize$v_b$}}
    % Counterclockwise half-edge labels at v_h.
    \put(220,330){\makebox(0,0)[c]{\scriptsize$\bar{1}$}}
    \put(240,370){\makebox(0,0)[c]{\scriptsize$\bar{2}$}}
    \put(245,410){\makebox(0,0)[c]{\scriptsize$\bar{3}$}}
    \put(235,445){\makebox(0,0)[c]{\scriptsize$\bar{4}$}}
    \put(200,485){\makebox(0,0)[c]{\scriptsize$\bar{5}$}}
    \put(70,330){\makebox(0,0)[c]{\scriptsize$\bar{6}$}}
    \put(120,305){\makebox(0,0)[c]{\scriptsize$\bar{7}$}}
    \put(185,310){\makebox(0,0)[c]{\scriptsize$\bar{8}$}}
    % Counterclockwise half-edge labels at v_a and v_b.
    \put(530,475){\makebox(0,0)[c]{\scriptsize$\bar{9}$}}
    \put(505,424){\makebox(0,0)[c]{\scriptsize$\bar{10}$}}
    \put(520,340){\makebox(0,0)[c]{\scriptsize$\bar{11}$}}
    \put(600,300){\makebox(0,0)[c]{\scriptsize$\bar{12}$}}
    \put(405,163){\makebox(0,0)[c]{\scriptsize$\bar{13}$}}
    \put(365,150){\makebox(0,0)[c]{\scriptsize$\bar{14}$}}
    \put(330,105){\makebox(0,0)[c]{\scriptsize$\bar{15}$}}
    \put(345,45){\makebox(0,0)[c]{\scriptsize$\bar{16}$}}
  \end{picture}
  \caption{The medial map $M(\Hcal_0)$ of $\Hcal_0$, with the canonical checkerboard coloring around its vertices.}
  \label{fig:example-medial}
\end{figure}
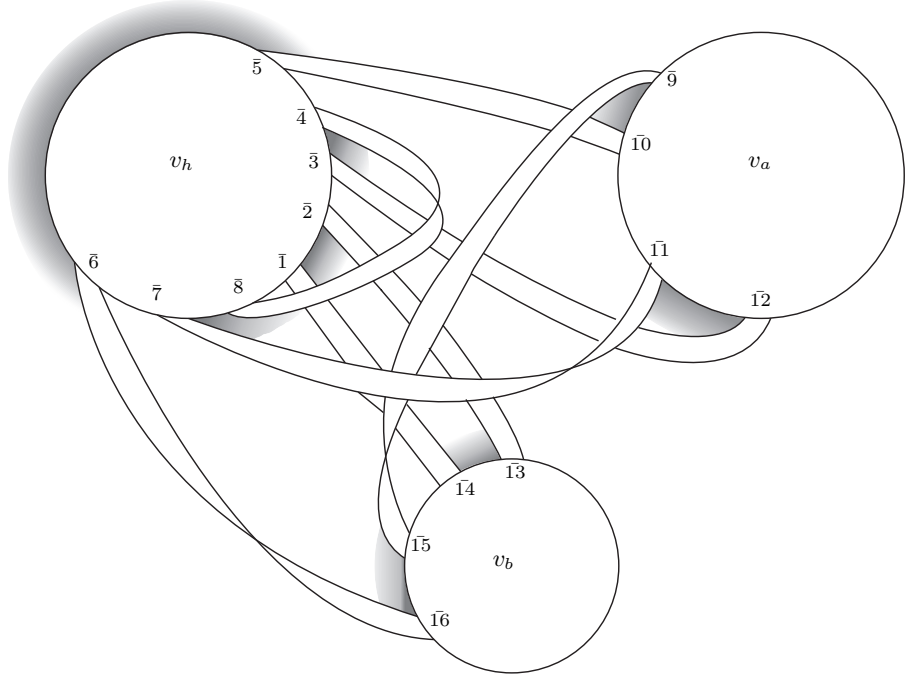

\begin{lemma}\label{lem:medial-incidences}
Let $Y/\langle\mu_2\rangle$ denote the set of orbits of \(\langle\mu_2\rangle\). Then the map
\[
  \iota:X\longrightarrow Y/\langle\mu_2\rangle,\qquad
  \iota(x)=[(x,\mathsf b)]_{\langle\mu_2\rangle}
  =\{(x,\mathsf b),(x,\mathsf w)\}
\]
is a bijection.

More precisely, let
\(\orb_{\langle\tau,\tau'\rangle}(x)\) be the orbit of the flag \(x\) under the group generated by \(\tau\) and \(\tau'\).

Then
\begin{align}
  \orb_{\langle\mu_1,\mu_2\rangle}(x,\mathsf b)
  &=e\times\{\mathsf b,\mathsf w\},\label{eq:medial-vertex-orbit}\\
  \orb_{\langle\mu_0,\mu_2\rangle}(x,\mathsf b)
  &=\{(x,\mathsf b),(x,\mathsf w),
      (\tau_1x,\mathsf b),(\tau_1x,\mathsf w)\},
      \label{eq:medial-edge-orbit}
\end{align}
where \(e=\orb_{\langle\tau_0,\tau_2\rangle}(x)\).
\end{lemma}

\begin{proof}
Since $\mu_2(x,\mathsf b)=(x,\mathsf w)$ and $\mu_2$ is fixed-point-free, $\iota$ is plainly a bijection.

From
\[
 \mu_1(x,\mathsf b)=(\tau_2x,\mathsf b),\qquad\mu_2\mu_1(x,\mathsf b)=(\tau_2x,\mathsf w),\qquad
 \mu_1\mu_2(x,\mathsf b)=(\tau_0x,\mathsf w),\qquad
 \mu_2\mu_1\mu_2(x,\mathsf b)=(\tau_0x,\mathsf b)
\]
the $\langle\mu_1,\mu_2\rangle$-orbit contains $e\times\{\mathsf b,\mathsf w\}$. Conversely, $\mu_1$ applies only $\tau_2$ on the $\mathsf b$-layer and only $\tau_0$ on the $\mathsf w$-layer, while $\mu_2$ does not change the first coordinate. Thus the orbit cannot leave $e\times\{\mathsf b,\mathsf w\}$, proving
\eqref{eq:medial-vertex-orbit}.

Moreover, $\mu_0$ and $\mu_2$ commute, and $\mu_0$ applies $\tau_1$ to the first coordinate, so~\eqref{eq:medial-edge-orbit} follows.
\end{proof}

\begin{proposition}\label{prop:tau-medial}
The $M(\Hcal)$ defined in Definition~\ref{def:tau-medial} is a map; that is, every $\langle\mu_0,\mu_2\rangle$-orbit contains exactly four flags. Its cells have the following natural correspondence with the data of $\Hcal$:
\[
\begin{array}{c|c}
\text{The $\tau$-data of $\Hcal$} & \text{Object in $M(\Hcal)$}\\ \hline
 e\in X/\langle\tau_0,\tau_2\rangle & \text{a medial vertex }v_e\\
 x\in e & \text{a medial half-edge at $v_e$ labelled by $x$}\\
 \{x,\tau_1x\} & \text{a medial edge}\\
 X/\langle\tau_1,\tau_2\rangle & \text{a black face}\\
 X/\langle\tau_1,\tau_0\rangle & \text{a white face}.
\end{array}
\]
In particular, if $e$ has valence $k$, then $v_e$ has degree $2k$. At $v_e$, the black angles project to pairs $\{x,\tau_2x\}$ and the white angles to pairs $\{x,\tau_0x\}$. Thus $M(\Hcal)$ has the canonical checkerboard colouring.
\end{proposition}

\begin{proof}
The three maps $\mu_i$ are fixed-point-free involutions. Starting with a flag $(x,\mathsf b)$, the maps $\mu_0$, $\mu_1$, and $\mu_2\mu_1\mu_2$ act on the first coordinate by $\tau_1$, $\tau_2$, and $\tau_0$, respectively, while $\mu_2$ changes the second coordinate. Hence the connectedness of $\Hcal$ implies that $M(\Hcal)$ is connected. Also, $\mu_0$ and $\mu_2$ commute. By~\eqref{eq:medial-edge-orbit}, every edge of $M(\Hcal)$ consists of four flags, so $M(\Hcal)$ is a map; moreover, its edges are in bijection with the $\tau_1$-transpositions. By~\eqref{eq:medial-vertex-orbit},
\[
 Y/\langle\mu_1,\mu_2\rangle
 \cong X/\langle\tau_0,\tau_2\rangle.
\]
This gives the correspondence between the medial vertices and the hyperedges of the original hypermap. Since $M(\Hcal)$ has been shown to be a map, $\mu_2$ preserves its vertices and edges while interchanging the two face-side flags; hence
$Y/\langle\mu_2\rangle$ is precisely its set of half-edges. By Lemma~\ref{lem:medial-incidences}, $\iota:X\to\mathcal I(M(\Hcal))$ is a bijection, and the half-edge set at $v_e$ is precisely $\{\iota(z):z\in e\}$. Therefore, if $\val_{\Hcal}(e)=k$, then
\[
  \deg_{M(\Hcal)}(v_e)=|e|=2k.
\]

Since $\mu_0$ preserves the second coordinate, whereas $\mu_1$ acts by $\tau_2$ on the $\mathsf b$-layer and by $\tau_0$ on the $\mathsf w$-layer,
\[
 Y_{\mathsf b}/\langle\mu_0,\mu_1\rangle
 \cong X/\langle\tau_1,\tau_2\rangle,
\]
and
\[
 Y_{\mathsf w}/\langle\mu_0,\mu_1\rangle
 \cong X/\langle\tau_1,\tau_0\rangle.
\]
Thus the former is a black face and the latter a white face. At a medial vertex, a black face turns according to the action of $\mu_1$ on the $\mathsf b$-layer, hence via a $\tau_2$-pair after projection; the white case is analogous, with $\tau_0$. This proves the assertions about angles and the checkerboard colouring. 
\end{proof}
Notice that, when the hypermap is orientable, the construction of the medial map given here coincides with the permutation construction of Cori--Hetyei~\cite[Definition~3.1]{CoriHetyei2025}.

The bijection $\iota$ of Lemma~\ref{lem:medial-incidences} says precisely that, after forgetting the second coordinate of $Y$, $X$ is the label set of the medial half-edges of $M(\Hcal)$. Geometrically, in the flag triangulation of $\Hcal$, take a small closed neighbourhood of each $02$-cell as a medial vertex disc and attach a narrow medial band across every colour-$1$ flag-graph edge $\{x,\tau_1x\}$; its two ends are labelled by $x$ and $\tau_1x$. Its black and white sides turn at vertex discs according to the $\tau_2$- and $\tau_0$-pairings, respectively, so the flag involutions of the resulting ribbon graph are exactly~\eqref{eq:tau-medial}. Thus $M(\Hcal)$ is not merely a combinatorial map but the geometric medial map of $\Hcal$. In the non-orientable case, an $02$-cycle provides only an unoriented cyclic order; the local orientation systems below record how bands are glued relative to the vertex discs.

\subsection{States}

Recall from Lemma~\ref{lem:medial-incidences} that every flag of the original hypermap corresponds to one medial half-edge. Although the state circles are drawn on the medial
map, we describe them using the flag set $X$ through this bijection.

The involutions $\mu_0,\mu_1,\mu_2$ act on the flag set $Y$ of the medial
map, whereas $\tau_0,\tau_1,\tau_2$ act on the flag set $X$ of the original
hypermap. Under the identification given by $\iota$, a
$\tau_1$-transposition is a medial edge. The involutions $\tau_0$ and
$\tau_2$ describe the two possible turns at a medial vertex.

We now explain which of these two turns is the black smoothing and which is
the white smoothing. By the definition of the medial map,
 $\mu_1(x,\mathsf b)=(\tau_2x,\mathsf b),\mu_1(x,\mathsf w)=(\tau_0x,\mathsf w)$. The two layers $\mathsf b$ and $\mathsf w$ represent the two sides of the
medial map. Furthermore, by Proposition~\ref{prop:tau-medial},
$ Y_{\mathsf b}/\langle\mu_0,\mu_1\rangle
 \cong X/\langle\tau_1,\tau_2\rangle$ is the set of black faces, while $Y_{\mathsf w}/\langle\mu_0,\mu_1\rangle
 \cong X/\langle\tau_1,\tau_0\rangle$ is the set of white faces.

A smoothing at a medial vertex joins two incident half-edges which are
connected through one side of the medial map. On the black side, the
connection is determined by $\tau_2$, since $\mu_1$ acts as $\tau_2$ on the
$\mathsf b$-layer. On the white side, the connection is determined by
$\tau_0$, since $\mu_1$ acts as $\tau_0$ on the $\mathsf w$-layer. Therefore,
\[
 \tau_2 \text{ corresponds to the black smoothing, whereas } \tau_0 \text{ corresponds to the white smoothing}.
\]

Define an involution $\eta_A$ on $X$ by
\[
 \eta_A(x)=
 \begin{cases}
   \tau_0(x),&x\in X_A,\\
   \tau_2(x),&x\notin X_A.
 \end{cases}
\]
The reason for this choice is the definition of hyperedge partial duality.
Notice that $\eta_A=\tau_2^A$ as involutions on the original flag set $X$.
Geometrically, this means that the state uses the white smoothing at every
hyperedge in $A$ and the black smoothing at every hyperedge outside $A$.

\begin{definition}\label{def:state}
For $A\subseteq\Ecal(\Hcal)$, the \emph{state} $S_A$ is the edge-coloured
multigraph on $X$ whose two edge classes are the transpositions of
$\tau_1$ and the transpositions of $\eta_A$.

Using the bijection $\iota$, the same graph may be regarded as a graph on the
half-edge set of the medial map. In this interpretation, the $\tau_1$-edges
are the medial edges, and the $\eta_A$-edges are the chosen smoothing
connections at medial vertices. The connected components of $S_A$ are
called \emph{state circles}.

If a $\tau_1$-transposition and an $\eta_A$-transposition have the same
endpoints, they are still regarded as two parallel edges of different
colours. For a state circle $C$, let $E_{\tau_1}(C)$ be the set of all
$\tau_1$-edges in $C$, and define
\[
 \ell(C):=\left|E_{\tau_1}(C)\right|.
\]
\end{definition}

Since both $\tau_1$ and $\eta_A$ are fixed-point-free, every flag of $S_A$
is incident with exactly one edge of each colour. Hence every connected
component of $S_A$ is a cycle with alternating edge colours. We allow
$\ell(C)=1$; in that case, $C$ is a $2$-cycle consisting of two parallel
edges of different colours.

\begin{theorem}\label{thm:state-vertices}
The state circles of $S_A$ are naturally in bijection with the vertices of
$\Hcal^A$. If $C$ is a state circle and $v_C$ is the corresponding vertex,
then
\[
 \val_{\Hcal^A}(v_C)=\ell(C).
\]
\end{theorem}

\begin{proof}
Since $\tau_2^A=\eta_A$, the vertices of $\Hcal^A$ are exactly the $\langle\tau_1,\eta_A\rangle$-orbits in $X$. Moreover, because the edges of $S_A$ are precisely the transpositions occurring in $\tau_1$ and $\eta_A$, these orbits coincide with the connected components of $S_A$. Hence the state circles are in natural bijection with the vertices of $\Hcal^A$.

Let $C$ be one such component. Since $\tau_1$ is fixed-point-free, its
transpositions partition the flags of $C$ into disjoint pairs. Hence
\[
 |C|=2\left|E_{\tau_1}(C)\right|.
\]
The flag set of the corresponding vertex $v_C$ is $C$. Therefore,
\[
 \val_{\Hcal^A}(v_C)
 =\frac{|C|}{2}
 =\left|E_{\tau_1}(C)\right|
 =\ell(C).
\]
\end{proof}

\begin{corollary}\label{cor:even-state}
The partial dual $\Hcal^A$ is Eulerian if and only if every state circle of
$S_A$ has even length.
\end{corollary}

\begin{proof}
By Theorem~\ref{thm:state-vertices}, the valence of each vertex of
$\Hcal^A$ is equal to the length of the corresponding state circle.
Therefore, all vertex valences are even if and only if all state circles have
even length.
\end{proof}

\section{Eulerian Partial Duals}

\subsection{Crossing-Total Orientations}

By Lemma~\ref{lem:medial-incidences}, we identify $X$ with the half-edge set of $M(\Hcal)$ via $\iota$. Let $\Omega$ be an orientation of the edges of $M(\Hcal)$. For $x\in X$ (the medial half-edge labelled by $x$), set
\[
 \delta_\Omega(x)=
 \begin{cases}
  +1,&\text{if the medial edge is directed away from the medial vertex at }x,\\
  -1,&\text{if the medial edge is directed towards the medial vertex at }x.
 \end{cases}
\]
Then $\delta_\Omega(\tau_1x)=-\delta_\Omega(x)$. A black angle is a source or sink precisely when $\delta_\Omega(\tau_2x)=\delta_\Omega(x)$, and a white angle is a source or sink precisely when $\delta_\Omega(\tau_0x)=\delta_\Omega(x)$.

\begin{definition}\label{def:crossing-total}
Let $e$ be a hyperedge, regarded as a medial vertex.
\begin{enumerate}[label=\textup{(\alph*)}]
\item If $\delta_\Omega$ is constant on $e$, then $e$ is of \emph{$t$-type}.
\item If $e$ is not of $t$-type and $\delta_\Omega(\tau_2x)=\delta_\Omega(x)$ for every $x\in e$, then $e$ is of \emph{$c$-type}.
\item If $e$ is not of $t$-type and $\delta_\Omega(\tau_0x)=\delta_\Omega(x)$ for every $x\in e$, then $e$ is of \emph{$d$-type}.
\end{enumerate}
If every hyperedge is of one of these three types, then $\Omega$ is called a \emph{crossing-total orientation}. Let $D(\Omega)$ and $T(\Omega)$ denote the sets of $d$-type and $t$-type hyperedges, respectively.
\end{definition}

The three types are pairwise disjoint. Indeed, if the two same-sign conditions in (b) and (c) both hold, the connectedness of the $02$-cycle of $e$ implies that $\delta_\Omega$ is constant on $e$.

\begin{theorem}\label{thm:eulerian}
Let $\Hcal$ be a hypermap and $A\subseteq\Ecal(\Hcal)$. Then $\Hcal^A$ is Eulerian if and only if there is a crossing-total orientation $\Omega$ of $M(\Hcal)$ and a set $T'\subseteq T(\Omega)$ such that
\[
 A=D(\Omega)\cup T'.
\]
\end{theorem}

\begin{proof}
Suppose first that $\Hcal^A$ is Eulerian. By Corollary~\ref{cor:even-state}, every state circle $C$ of $S_A$ has even length. Choose a flag $x_0$ on each state circle and assign it the sign $+1$. Propagate signs according to
\[
  \delta(\eta_Ax)=\delta(x),\qquad
  \delta(\tau_1x)=-\delta(x).
  \tag{E}\label{eq:eulerian-sign}
\]
When one traverses $C$ once, exactly $\ell(C)$ $\tau_1$-edges are crossed, so the cycle relation gives the multiplier $(-1)^{\ell(C)}=1$; hence the assignment is well-defined.

Each $\tau_1$-transposition $\{x,\tau_1x\}$ corresponds to a medial edge. Equation~\eqref{eq:eulerian-sign} gives $\delta(\tau_1x)=-\delta(x)$. Thus the medial edge has a unique orientation for which it is directed away from the corresponding medial vertex at the endpoint labelled $x$ if and only if $\delta(x)=+1$. This yields an orientation $\Omega$ of $M(\Hcal)$.

If $e\in A$, then $\eta_A|_e=\tau_0|_e$, so for every $x\in e$,
\[
  \delta(\tau_0x)=\delta(x).
\]
If $\delta$ is constant on $e$, then $e$ is of $t$-type; otherwise it is of $d$-type. If $e\notin A$, then $\eta_A|_e=\tau_2|_e$, so for every $x\in e$,
\[
  \delta(\tau_2x)=\delta(x).
\]
If $\delta$ is constant on $e$, then $e$ is of $t$-type; otherwise it is of $c$-type. Thus $\Omega$ is crossing-total. Since the three types are pairwise disjoint,
\[
  D(\Omega)\subseteq A,\qquad
  A\setminus D(\Omega)\subseteq T(\Omega).
\]
Set
\[
  T':=A\setminus D(\Omega)
\]
Then $T'\subseteq T(\Omega)$ and $A=D(\Omega)\cup T'$.

Conversely, suppose that $\Omega$ is crossing-total and
\[
  A=D(\Omega)\cup T',\qquad T'\subseteq T(\Omega).
\]
Let $\delta_\Omega$ be the endpoint-sign function of this orientation. For $x\in X$, let $e(x)$ be the hyperedge containing $x$. If $e(x)\in A$, then $e(x)$ belongs either to $D(\Omega)$ or to $T'$. In the first case $e(x)$ is of $d$-type and $\eta_A=\tau_0$, while in the second it is of $t$-type. Thus in both cases
\[
  \delta_\Omega(\eta_Ax)=\delta_\Omega(x).
\]
If $e(x)\notin A$, then it does not belong to $D(\Omega)$. Since $\Omega$ is crossing-total and the three types are pairwise disjoint, $e(x)$ is of $c$-type or $t$-type; now $\eta_A=\tau_2$, so again
\[
  \delta_\Omega(\eta_Ax)=\delta_\Omega(x).
\]
On the other hand, every directed medial edge satisfies
\[
  \delta_\Omega(\tau_1x)=-\delta_\Omega(x).
\]

Now let $C$ be any state circle of $S_A$. Starting from a flag of $C$ and walking around it, the sign is unchanged across an $\eta_A$-edge and changes across a $\tau_1$-edge. Returning to the initial flag after one circuit gives
\[
  \delta_\Omega(x)=(-1)^{\ell(C)}\delta_\Omega(x).
\]
Since $\delta_\Omega(x)\in\{+1,-1\}$, $\ell(C)$ must be even. Hence every state circle of $S_A$ has even length, and Corollary~\ref{cor:even-state} implies that $\Hcal^A$ is Eulerian.
\end{proof}

\section{Modified Medial Maps and All-Crossing Orientations}

The preceding proof does not use a global orientation. Bipartiteness is different: the usual left--right labels of a directed medial edge cannot be propagated consistently along a path that passes through twisted medial bands. We replace these labels by local orientations on vertex discs and the band-gluing type relative to that choice.

\subsection{Relative Twist Data}

\begin{definition}\label{def:local-orientation}
A \emph{local orientation system} for $M(\Hcal)$ is a map $\omega:X\to\{+1,-1\}$ satisfying
\[
 \omega(\tau_0x)=-\omega(x),\qquad \omega(\tau_2x)=-\omega(x)
 \qquad (x\in X).
\]
\end{definition}

Such a system always exists: choose a sign on each $02$-orbit and propagate it around the alternating cycle. On a fixed hyperedge there are two choices, differing only by a common sign.

\begin{lemma}\label{lem:local-orientation-flags}
For any local orientation system $\omega$, define $q_\omega:Y\to\{+1,-1\}$ by
\[
  q_\omega(x,\mathsf b)=\omega(x),\qquad
  q_\omega(x,\mathsf w)=-\omega(x).
\]
Then
\[
  q_\omega(\mu_1y)=-q_\omega(y),\qquad
  q_\omega(\mu_2y)=-q_\omega(y)\qquad(y\in Y). \tag{L}\label{eq:local-orientation-flags}
\]
Conversely, every map $q:Y\to\{+1,-1\}$ satisfying~\eqref{eq:local-orientation-flags} arises uniquely from a local orientation system.

In the flag triangulation of $M(\Hcal)$, $q_\omega$ is precisely the flag-sign representation of the local orientation of each medial vertex disc: on every $\langle\mu_1,\mu_2\rangle$-orbit, the restriction of $q_\omega$ and its negative correspond to the two local orientations of that vertex disc.
\end{lemma}

\begin{proof}
The second equality follows immediately from $\mu_2(x,\mathsf c)=(x,\bar{\mathsf c})$, where $\{\mathsf c,\bar{\mathsf c}\}=\{\mathsf b,{\mathsf w}\}$. If $y=(x,\mathsf b)$, then
\[
 q_\omega(\mu_1y)
 =q_\omega(\tau_2x,\mathsf b)
 =\omega(\tau_2x)=-\omega(x)=-q_\omega(y).
\]
If $y=(x,\mathsf w)$, then
\[
 q_\omega(\mu_1y)
 =q_\omega(\tau_0x,\mathsf w)
 =-\omega(\tau_0x)=\omega(x)=-q_\omega(y).
\]
Thus~\eqref{eq:local-orientation-flags} holds.

Conversely, let $\omega(x)=q(x,\mathsf b)$. The $\mu_2$ relation gives $q(x,\mathsf w)=-\omega(x)$; applying the $\mu_1$ relation at $(x,\mathsf b)$ and $(x,\mathsf w)$, respectively, gives
\[
  \omega(\tau_2x)=-\omega(x),\qquad
  \omega(\tau_0x)=-\omega(x).
\]
Thus $\omega$ is a local orientation system, and uniqueness is clear.

Finally, the flags of a medial vertex disc form exactly its $\langle\mu_1,\mu_2\rangle$-orbit, whose flag triangulation is a disc. Adjacent small triangles meet across colour-$1$ or colour-$2$ edges; with the colour order fixed, an orientation of the disc is exactly one for which adjacent triangles have opposite flag signs. Conversely, the alternating signs in~\eqref{eq:local-orientation-flags} make the orientations of these triangles compatible; since the complex is a disc, they determine a unique local orientation. Changing all signs reverses the orientation of the disc.
\end{proof}

For a medial edge $a=\{x,\tau_1x\}$, define its \emph{twist indicator} relative to $\omega$ by $\varepsilon_\omega(a)\in\{0,1\}$,
\[
 \omega(\tau_1x)=(-1)^{1+\varepsilon_\omega(a)}\omega(x).
 \tag{T}\label{eq:twist}
\]
This definition is independent of the choice of endpoint $x$: exchanging $x$ and $\tau_1x$ replaces the ratio of the two values in $\{+1,-1\}$ by its reciprocal, which is the same ratio.

\begin{lemma}\label{lem:twist-flag}
Let $a=\{x,\tau_1x\}$ be a medial edge. Then for $c\in\{\mathsf b,\mathsf w\}$,
\[
 q_\omega\bigl(\mu_0(x,c)\bigr)
 =(-1)^{1+\varepsilon_\omega(a)}q_\omega(x,c). \tag{U}
 \label{eq:twist-from-flags}
\]
Thus, relative to the local orientations encoded by $q_\omega$, $\varepsilon_\omega(a)=0$ if and only if the band $a$ is untwisted, while $\varepsilon_\omega(a)=1$ if and only if the band is twisted.

If $\omega'$ is another local orientation system and
\[
 \omega'(x)=r(e(x))\omega(x),\qquad
 r(e)=(-1)^{\rho(e)},
\]
where $e(x)$ is the hyperedge containing $x$, then
\[
 \varepsilon_{\omega'}(a)
 \equiv\varepsilon_\omega(a)+\rho(e(x))
 +\rho(e(\tau_1x))\pmod 2. \tag{S}\label{eq:twist-switch}
\]
\end{lemma}

\begin{proof}
For $c=\mathsf b$,
\[
 q_\omega(\mu_0(x,\mathsf b))
 =q_\omega(\tau_1x,\mathsf b)
 =\omega(\tau_1x)
 =(-1)^{1+\varepsilon_\omega(a)}q_\omega(x,\mathsf b).
\]
For $c=\mathsf w$, both sides acquire an additional minus sign, so the same equality remains valid.

Represent the band by a rectangle. The flags $(x,c)$ and $\mu_0(x,c)=(\tau_1x,c)$ lie at the two ends on the same side of the rectangle. With the colour order fixed, the local orientations of the endpoint discs extend compatibly across the rectangle if and only if the two signs across this colour-$0$ flag edge are opposite; this is exactly
\[
  q_\omega(\mu_0(x,c))=-q_\omega(x,c).
\]
If the two signs agree, transporting a local orientation along the band produces the opposite orientation, so the band is twisted relative to the chosen local orientations (with twisting taken modulo $2$). Together with~\eqref{eq:twist-from-flags}, this proves the criterion. The criterion is independent of the choice of $c$.

For $y=(x,c)$ we also have
\[
 q_{\omega'}(y)=r(e(x))q_\omega(y).
\]
Thus
\[
\begin{aligned}
 q_{\omega'}(\mu_0y)
 &=r(e(\tau_1x))q_\omega(\mu_0y)\\
 &=(-1)^{1+\varepsilon_\omega(a)}
   r(e(\tau_1x))r(e(x))q_{\omega'}(y).
\end{aligned}
\]
Substituting $r=(-1)^\rho$ and comparing with~\eqref{eq:twist-from-flags} gives~\eqref{eq:twist-switch}.

If $a$ is a loop, then $e(x)=e(\tau_1x)$ and the additional term in the formula is $2\rho(e(x))\equiv0\pmod2$. Thus reversing the same vertex disc reverses both endpoint local orientations but does not change the twist of a loop band, in agreement with its topological definition.
\end{proof}

Lemma~\ref{lem:twist-flag} shows that the twist indicator is not an intrinsic $0$--$1$ label: when the local orientation of a medial vertex disc is changed, the band indicators transform according to~\eqref{eq:twist-switch}. This is exactly the switching operation in a signed rotation representation. Thus the underlying medial map is determined by the $\tau$-model, while $(\omega,\varepsilon_\omega)$ is only one local signed description of it.

\begin{definition}\label{def:modified-medial}
For a local orientation system $\omega$, the \emph{modified medial map} $\widehat M_\omega(\Hcal)$ is obtained from $M(\Hcal)$ by inserting an artificial bivalent bar into each medial edge $a$ satisfying $\varepsilon_\omega(a)=1$. The checkerboard colouring of $M(\Hcal)$ extends across every subdivision.
\end{definition}

Figure~\ref{fig:modified-medial} illustrates this local operation: when a medial band is twisted relative to the chosen local orientations, an artificial bivalent bar is inserted in its interior, splitting the original edge into two segments. This artificial bar represents neither a hyperedge of the original hypermap nor an original medial vertex; it records only the twisting obstruction of the band. In the orientation definitions below, this bivalent bar is required to be a source or a sink.

\begin{figure}[htbp]
  \centering
  \includegraphics[width=10cm]{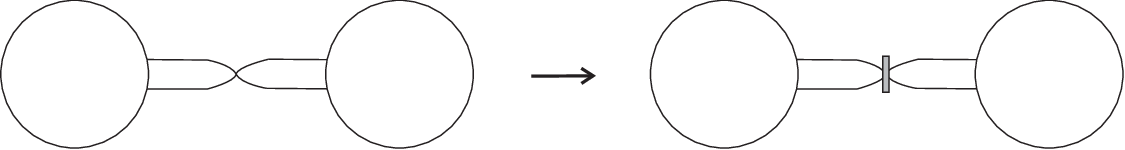}
  \caption{Local construction of the modified medial map.}
  \label{fig:modified-medial}
\end{figure}

Thus the set of subdivided medial edges generally depends on $\omega$. The following lemma shows, however, that the existence of an all-crossing orientation relevant to bipartite partial duals, and the associated set $C(\Phi)$, do not depend on this local choice.

The following lemma provides a convenient algebraic encoding of such orientations.

\begin{lemma}\label{lem:signed-directions}
Orientations of $\widehat M_\omega(\Hcal)$ for which every artificial bivalent bar is a source or a sink are in bijection with maps $\delta:X\to\{+1,-1\}$ satisfying
\[
 \delta(\tau_1x)=(-1)^{1+\varepsilon_\omega(\{x,\tau_1x\})}\delta(x)
 \qquad(x\in X).\tag{D}\label{eq:signed-direction}
\]
Here $\delta(x)=+1$ means that the edge segment at $x$ is directed away from its incident original medial vertex.
\end{lemma}

\begin{proof}
If $\varepsilon_\omega(a)=0$, the edge $a$ is not subdivided and its endpoint signs are opposite, as for an ordinary directed edge. If $\varepsilon_\omega(a)=1$, the two segments meet at an artificial source or sink, and their signs at the two original medial vertices agree. These two local cases give exactly~\eqref{eq:signed-direction}, and the construction is reversible.
\end{proof}

\begin{definition}\label{def:all-crossing}
Let $\Phi$ be an orientation encoded by $\delta$ as in Lemma~\ref{lem:signed-directions}. If, for every $x\in e$,
\[
 \delta(\tau_2x)=\delta(x),\qquad
 \delta(\tau_0x)=-\delta(x),
\]
then $e$ is called a \emph{$c$-type all-crossing vertex}; if, for every $x\in e$,
\[
 \delta(\tau_0x)=\delta(x),\qquad
 \delta(\tau_2x)=-\delta(x),
\]
then it is called a \emph{$d$-type all-crossing vertex}. If every original medial vertex is of exactly one of these two types, then $\Phi$ is called \emph{all-crossing}. Let $C(\Phi)$ denote the set of its $c$-type hyperedges.
\end{definition}

Thus, at a $c$-type vertex all black angles are sources or sinks and all white angles are crossings; at a $d$-type vertex the roles of black and white are interchanged.

To connect these definitions with the running example, Figure~\ref{fig:all-crossing-example} shows an all-crossing orientation on the medial map in Figure~\ref{fig:example-medial}. The four small discs in the figure lie inside the twisted medial edges $m_{4,8}$, $m_{7,11}$, $m_{6,16}$, and $m_{9,15}$, and therefore represent the four artificial bivalent bars of $\widehat M_\omega(\Hcal_0)$. The arrows give the directions of the edge segments and form a source or sink at each artificial bar. In the half-edge orders at $v_h,v_a,v_b$, the endpoint signs are
\[
\delta=(+,+,-,-,+,+,-,-;\,+,-,-,+;\,-,-,+,+).
\]
Thus $h$ and $b$ are $c$-type hyperedges and $a$ is $d$-type, so the set of $c$-type hyperedges is $C(\Phi)=\{h,b\}$. This gives a concrete example of the set $A=C(\Phi)$ in Theorem~\ref{thm:bipartite}.

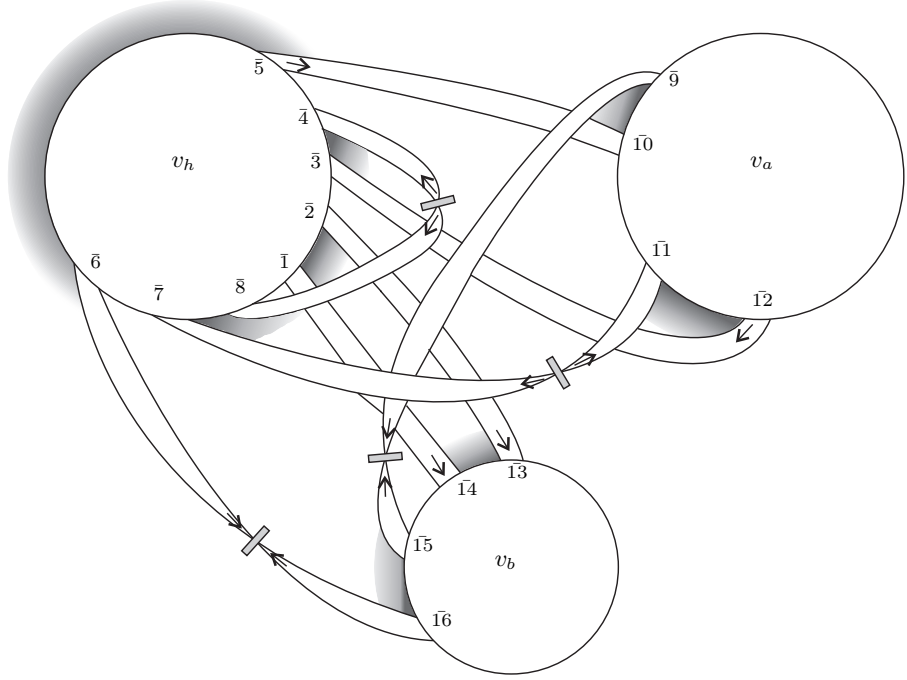
\begin{figure}[htbp]
  \centering
  \setlength{\unitlength}{0.001\textwidth}
    \begin{picture}(720,541)
    \put(0,0){\includegraphics[width=0.72\textwidth]{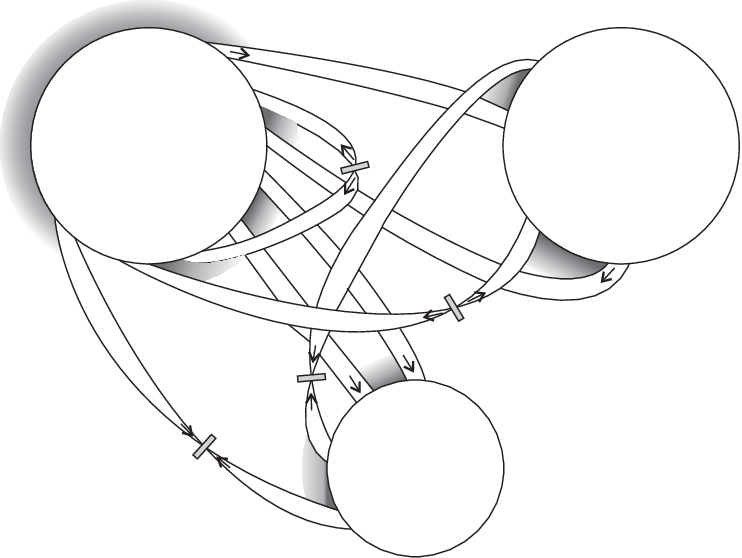}}
    \put(140,408){\makebox(0,0)[c]{\footnotesize$v_h$}}
    \put(600,408){\makebox(0,0)[c]{\footnotesize$v_a$}}
    \put(396,90){\makebox(0,0)[c]{\footnotesize$v_b$}}
    % Counterclockwise half-edge labels at v_h.
    \put(220,330){\makebox(0,0)[c]{\scriptsize$\bar{1}$}}
    \put(240,370){\makebox(0,0)[c]{\scriptsize$\bar{2}$}}
    \put(245,410){\makebox(0,0)[c]{\scriptsize$\bar{3}$}}
    \put(235,445){\makebox(0,0)[c]{\scriptsize$\bar{4}$}}
    \put(200,485){\makebox(0,0)[c]{\scriptsize$\bar{5}$}}
    \put(70,330){\makebox(0,0)[c]{\scriptsize$\bar{6}$}}
    \put(120,305){\makebox(0,0)[c]{\scriptsize$\bar{7}$}}
    \put(185,310){\makebox(0,0)[c]{\scriptsize$\bar{8}$}}
    % Counterclockwise half-edge labels at v_a and v_b.
    \put(530,475){\makebox(0,0)[c]{\scriptsize$\bar{9}$}}
    \put(505,424){\makebox(0,0)[c]{\scriptsize$\bar{10}$}}
    \put(520,340){\makebox(0,0)[c]{\scriptsize$\bar{11}$}}
    \put(600,300){\makebox(0,0)[c]{\scriptsize$\bar{12}$}}
    \put(405,163){\makebox(0,0)[c]{\scriptsize$\bar{13}$}}
    \put(365,150){\makebox(0,0)[c]{\scriptsize$\bar{14}$}}
    \put(330,105){\makebox(0,0)[c]{\scriptsize$\bar{15}$}}
    \put(345,45){\makebox(0,0)[c]{\scriptsize$\bar{16}$}}
  \end{picture}
  \caption{An all-crossing orientation of $\Hcal_0$. }
  \label{fig:all-crossing-example}
\end{figure}

\begin{remark}\label{rem:high-valence}
If a hyperedge $e$ has valence $k$, then the corresponding medial vertex $v_e$ has degree $2k$. Definition~\ref{def:crossing-total} uses the three types $c,d,t$, whereas Definition~\ref{def:all-crossing} uses only $c,d$. When $k>2$, our $c$- and $d$-types are definitions for higher-valence medial vertices: $c$-type requires every black angle to be a source or sink and every white angle to be a crossing, while $d$-type interchanges the roles of black and white. We do not apply the four-valent ``in,in,out,out'' pattern to higher-valence vertices without qualification. Only when $k=2$ is $v_e$ four-valent, and our $c,d,t$ types are exactly the classical local patterns in a medial graph.
\end{remark}

\begin{lemma}\label{lem:choice-independent}
The existence of an all-crossing orientation with a prescribed set $C(\Phi)$ is independent of the choice of local orientation system $\omega$.
\end{lemma}

\begin{proof}
Let $\omega'$ be another local orientation system. Since the ratio $\omega'/\omega$ is invariant under both $\tau_0$ and $\tau_2$, it is constant on each hyperedge. Thus there is a map $r:\Ecal(\Hcal)\to\{+1,-1\}$ such that $\omega'(x)=r(e(x))\omega(x)$, where $e(x)$ is the hyperedge containing $x$. If $\delta$ satisfies~\eqref{eq:signed-direction}, set $\delta'(x)=r(e(x))\delta(x)$. For $a=\{x,\tau_1x\}$, equations~\eqref{eq:signed-direction} and~\eqref{eq:twist-switch} give
\[
\begin{aligned}
 \delta'(\tau_1x)
 &=r(e(\tau_1x))(-1)^{1+\varepsilon_\omega(a)}\delta(x)\\
 &=(-1)^{1+\varepsilon_{\omega'}(a)}r(e(x))\delta(x)\\
 &=(-1)^{1+\varepsilon_{\omega'}(a)}\delta'(x).
\end{aligned}
\]
Thus $\delta'$ satisfies the corresponding form of~\eqref{eq:signed-direction} for $\omega'$. Applying the same transformation with $r$ recovers $\delta$, so this gives a bijection between the two classes of orientations. For a pair of $\tau_0$- or $\tau_2$-paired flags, both flags lie in the same hyperedge and are multiplied by the same factor $r(e)$. Consequently all same-sign and opposite-sign conditions in Definition~\ref{def:all-crossing} are preserved, and in particular $C(\Phi)$ is unchanged.
\end{proof}

\section{Bipartite Partial Duals}

\begin{theorem}\label{thm:bipartite}
Let $\Hcal$ be a hypermap, and let $A\subseteq\Ecal(\Hcal)$. Then the partial dual $\Hcal^A$ is bipartite if and only if, for any local orientation system $\omega$, the modified medial map $\widehat M_\omega(\Hcal)$ admits an all-crossing orientation $\Phi$ such that
\[
A=C(\Phi).
\]
\end{theorem}

\begin{proof} By Lemma~\ref{lem:choice-independent}, this condition is independent of the chosen local orientation system $\omega$.

Suppose first that $\Hcal^A$ is bipartite. Let $\lambda:X\to\{+1,-1\}$ be the pullback of a bipartition of its vertex-adjacency multigraph. Since the vertices of $\Hcal^A$ are the $\langle\tau_1,\tau_2^A\rangle$-orbits,
\[
 \lambda(\tau_1x)=\lambda(\tau_2^Ax)=\lambda(x),
 \qquad
 \lambda(\tau_0^Ax)=-\lambda(x).\tag{B}\label{eq:bipartition-identities}
\]
Set $\delta(x)=\lambda(x)\omega(x)$. The first equality in~\eqref{eq:bipartition-identities}, together with~\eqref{eq:twist}, implies~\eqref{eq:signed-direction}; hence $\delta$ gives an orientation of the modified medial map.

If $e\in A$, then on $e$ we have $\tau_2^A=\tau_0$ and $\tau_0^A=\tau_2$. Using the identities in~\eqref{eq:bipartition-identities} and the fact that $\omega$ changes sign under both $\tau_0$ and $\tau_2$, we obtain
\[
 \delta(\tau_0x)=-\delta(x),\qquad
 \delta(\tau_2x)=\delta(x)\qquad(x\in e).
\]
Thus $e$ is of $c$-type. If $e\notin A$, the same calculation gives
\[
 \delta(\tau_0x)=\delta(x),\qquad
 \delta(\tau_2x)=-\delta(x)\qquad(x\in e),
\]
so $e$ is of $d$-type. The resulting orientation is therefore all-crossing, and $C(\Phi)=A$.

Conversely, let $\Phi$ be an all-crossing orientation with $C(\Phi)=A$, and let $\delta$ be its signed direction map. Define
\[
 \lambda(x)=\delta(x)\omega(x).
\]
Equations~\eqref{eq:twist} and~\eqref{eq:signed-direction} show that $\lambda(\tau_1x)=\lambda(x)$. If $e\in A$, the $c$-type conditions give
\[
 \lambda(\tau_0x)=\lambda(x),\qquad
 \lambda(\tau_2x)=-\lambda(x).
\]
Since $\tau_2^A=\tau_0$ and $\tau_0^A=\tau_2$ on $e$, this yields
\[
 \lambda(\tau_2^Ax)=\lambda(x),\qquad
 \lambda(\tau_0^Ax)=-\lambda(x).
\]
If $e\notin A$, the $d$-type conditions give the same two identities, since now $\tau_2^A=\tau_2$ and $\tau_0^A=\tau_0$. Hence $\lambda$ is constant on every vertex orbit of $\Hcal^A$ and changes sign along every adjacency edge of $\Gamma(\Hcal^A)$. It is therefore a valid bipartition of $\Hcal^A$.
\end{proof}
\begin{example}\label{ex:mixed-partials}
It is easy to check that $\Hcal_0$ has exactly one all-crossing orientation, as shown in Figure~\ref{fig:all-crossing-example}.
Let $A=\{h,b\}$ denote the set of $c$-type edges identified by the all-crossing orientation. The identity $\mathcal V(\Hcal_0^A)=X/\langle\tau_1,\tau_2^A\rangle$ directly yields the two vertex orbits
\[
 P=(1,4,5,8,9,10,14,15),\qquad
 Q=(2,3,6,7,11,12,13,16).
\]
In $\Gamma(\Hcal_0^A)$, every adjacency edge contributed by $h$, $a$, and $b$ joins $P$ to $Q$. Hence $\{P\}\sqcup\{Q\}$ is a bipartition, so $\Hcal_0^A$ is the bipartite partial dual of $\Hcal_0$.
\end{example}

\begin{corollary}\label{cor:even-obstruction}
If $\Hcal^A$ is bipartite for some $A\subseteq\Ecal(\Hcal)$, then every hyperedge of $\Hcal$ has even valence.
\end{corollary}

\begin{proof}
By Theorem~\ref{thm:bipartite}, every hyperedge in an all-crossing orientation is of $c$- or $d$-type. Around a $c$-type hyperedge, the sign $\delta$ changes across every $\tau_0$-pair and is unchanged across every $\tau_2$-pair. Thus, on traversing the alternating $02$-cycle once, the sign changes once for every $\tau_0$-pair. It must return to its initial value, so the number of such pairs, namely the valence of the hyperedge, is even. The $d$-type case is identical with the roles of $\tau_0$ and $\tau_2$ exchanged.
\end{proof}

\begin{example}\label{ex:odd-hyperedge}
The projective-plane hypermap depicted in Figure~\ref{fig:tau-involutions} has a hyperedge of valence three. Corollary~\ref{cor:even-obstruction} immediately implies that neither of its four hyperedge partial duals is bipartite.
\end{example}

\begin{remark}\label{prop:specializations}
In the following two special cases, our constructions reduce to existing orientation characterizations.
\begin{enumerate}[label=\textup{(\roman*)}]
\item If $\Hcal$ is orientable, choose a global flag-sign function $\omega$ such that
\[
 \omega(\tau_i x)=-\omega(x)\qquad(i=0,1,2).
\]
Then $\varepsilon_\omega(a)=0$ for every medial edge $a$, and hence $\widehat M_\omega(\Hcal)=M(\Hcal)$. Under the standard correspondence between the $\tau$- and $(\sigma,\alpha)$-models (see~\cite[Section~1.3 and Theorem~2.4]{ChmutovVignes2022}; the permutation construction of the medial map is in~\cite[Definition~3.1]{CoriHetyei2025}), Theorems~\ref{thm:eulerian} and~\ref{thm:bipartite} reduce respectively to~\cite[Theorems~5.3 and~6.5]{HanMetsidik2026}.

\item If every hyperedge $e$ satisfies $\val(e)=2$, then every $\langle\tau_0,\tau_2\rangle$-orbit has four flags, so $\Hcal$ is a map and determines a ribbon graph $G$. Under the natural correspondence between hyperedges and edges of $G$, $M(\Hcal)$ is isomorphic, as a checkerboard-coloured map, to the usual medial graph $G_m$ of $G$; every $v_e$ is four-valent. In this case, $t$-type means that all four incident half-edges point in or all point out, while $c$- and $d$-types are the two classical all-crossing patterns in which the black or white angles are sources/sinks and the other colour consists of crossings.

Furthermore, by Definition~\ref{def:partial-dual}, under the natural correspondence between hyperedges and edges of $G$, $\Hcal^A$ is exactly the ribbon-graph partial dual $G^A$: both swap colours $0$ and $2$ in the four-flag $02$-orbit of each selected edge. By Lemma~\ref{lem:twist-flag}, $\omega$ is a choice of local orientations for the vertex discs of $G_m$, and $\varepsilon_\omega(a)=1$ precisely when $a$ is twisted relative to this choice. Thus $\widehat M_\omega(\Hcal)$ is isomorphic to the modified medial graph $\widehat G_m$ obtained from the same local orientation system. Theorem~\ref{thm:bipartite} therefore reduces to~\cite[Theorem~1.4]{DengJinMetsidik2020}, and Theorem~\ref{thm:eulerian} to~\cite[Theorem~1.5]{DengJinMetsidik2020}.
\end{enumerate}
\end{remark}

\begin{proof}
For (i), orientability is equivalent to a bipartition of the flag graph in which every $\tau_i$ changes the sign; see~\cite[Section~1.3]{ChmutovVignes2022}. Equation~\eqref{eq:twist} then gives $\varepsilon_\omega=0$. For (ii), the four-flag $02$-orbit is exactly the criterion for a map; the flag correspondence for the medial map is given by Proposition~\ref{prop:tau-medial}, and Lemma~\ref{lem:twist-flag} identifies band twists with the modified medial graph under the same local orientations. The remaining assertions about local types follow immediately from the alternating black and white angles at a four-valent vertex.
\end{proof}

\section{Conclusion}

The $\tau$-model renders both characterizations independent of orientability. The state-circle argument identifies the common mechanism behind the Eulerian theorem, while local orientation systems encode precisely the extra data needed for bipartiteness. Because vertices, edges, and faces play symmetric roles in a hypermap, the analogous characterizations for vertex and face partial duals are obtained simply by interchanging the subscripts of $\tau_i$.

\section*{Declaration of competing interest}
We declare that we have no financial and personal relationships with other people or organizations that can inappropriately influence our work.

\section*{Acknowledgements}
This work is supported by National Natural Science Foundation of China (Grant Number: 12661073) and Natural Science Foundation of Xinjiang (Grant Number: 2024D01A89).

\section*{Data availability}
No data was used for the research described in the article.
\bibliographystyle{unsrt}
\bibliography{Eulerian_and_Bipartite_Partial_Duals_of_Hypermaps}

\begin{thebibliography}{10}

\bibitem{Chmutov2009}
S.~Chmutov.
\newblock {Generalized duality for graphs on surfaces and the signed
  Bollobas--Riordan polynomial}.
\newblock {\em J. Combin. Theory Ser. B}, 99(3):617--638, 2009.

\bibitem{ChmutovVignes2022}
S.~Chmutov and F.~Vignes-Tourneret.
\newblock {Partial duality of hypermaps}.
\newblock {\em Arnold Math. J.}, 8:445--468, 2022.

\bibitem{Moffatt2011}
I.~Moffatt.
\newblock {A characterization of partially dual graphs}.
\newblock {\em J. Graph Theory}, 67(3):198--217, 2011.

\bibitem{Moffatt2013}
I.~Moffatt.
\newblock {Separability and the genus of a partial dual}.
\newblock {\em European J. Combin.}, 34(2):355--378, 2013.

\bibitem{HuggettMoffatt2013}
S.~Huggett and I.~Moffatt.
\newblock {Bipartite partial duals and circuits in medial graphs}.
\newblock {\em Combinatorica}, 33:231--252, 2013.

\bibitem{MetsidikJin2018}
M.~Metsidik and X.~Jin.
\newblock {Eulerian partial duals of plane graphs}.
\newblock {\em J. Graph Theory}, 87(4):509--515, 2018.

\bibitem{DengJinMetsidik2020}
Q.~Deng, X.~Jin, and M.~Metsidik.
\newblock {Characterizations of bipartite and Eulerian partial duals of ribbon
  graphs}.
\newblock {\em Discrete Math.}, 343(1):111637, 2020.

\bibitem{Cori1975}
R.~Cori.
\newblock {\em {Un code pour les graphes planaires et ses applications}},
  volume~27 of {\em Asterisque}.
\newblock Societe Mathematique de France, Paris, 1975.

\bibitem{Walsh1975}
T.~R.~S. Walsh.
\newblock {Hypermaps versus bipartite maps}.
\newblock {\em J. Combin. Theory Ser. B}, 18(2):155--163, 1975.

\bibitem{JonesSingerman1978}
G.~A. Jones and D.~Singerman.
\newblock {Theory of maps on orientable surfaces}.
\newblock {\em Proc. London Math. Soc.}, 37(2):273--307, 1978.

\bibitem{LandoZvonkin2004}
S.~K. Lando and A.~K. Zvonkin.
\newblock {\em {Graphs on Surfaces and Their Applications}}.
\newblock Springer, Berlin, 2004.

\bibitem{Lins1982}
S.~Lins.
\newblock {Graph-encoded maps}.
\newblock {\em J. Combin. Theory Ser. B}, 32(2):171--181, 1982.

\bibitem{Vince1983}
A.~Vince.
\newblock {Combinatorial maps}.
\newblock {\em J. Combin. Theory Ser. B}, 34(1):1--21, 1983.

\bibitem{GrossTucker1987}
J.~L. Gross and T.~W. Tucker.
\newblock {\em {Topological Graph Theory}}.
\newblock Wiley, New York, 1987.

\bibitem{EllisMonaghanMoffatt2013}
J.~A. Ellis-Monaghan and I.~Moffatt.
\newblock {\em {Graphs on Surfaces: Dualities, Polynomials, and Knots}}.
\newblock Springer, New York, 2013.

\bibitem{CoriHetyei2025}
R.~Cori and G.~Hetyei.
\newblock {A Whitney polynomial for hypermaps}.
\newblock {\em Adv. Appl. Math.}, 171:102951, 2025.

\bibitem{HanMetsidik2026}
Y.~Han and M.~Metsidik.
\newblock {Characterizations of bipartite and Eulerian partial duals of
  orientable hypermaps}.
\newblock \href{https://arxiv.org/abs/2606.30071}{arXiv:2606.30071 [math.CO]},
  2026.

\bibitem{EllisMonaghanMoffatt2012}
J.~A. Ellis-Monaghan and I.~Moffatt.
\newblock {Twisted duality for embedded graphs}.
\newblock {\em Trans. Amer. Math. Soc.}, 364(3):1529--1569, 2012.

\bibitem{JacksonVisentin1999}
D.~M. Jackson and T.~I. Visentin.
\newblock {A combinatorial relationship between Eulerian maps and hypermaps in
  orientable surfaces}.
\newblock {\em J. Combin. Theory Ser. A}, 87(1):120--150, 1999.

\bibitem{BredaDuarte2007}
A.~Breda d'Azevedo and R.~Duarte.
\newblock {Bipartite-uniform hypermaps on the sphere}.
\newblock {\em Electron. J. Combin.}, 14(1):R5, 2007.

\end{thebibliography}

\end{document}